\documentclass{amsart}
\usepackage[utf8]{inputenc}
\usepackage{amsmath, amssymb}
\usepackage{amsthm}
\usepackage[all]{xy}
\usepackage{xcolor}
\definecolor{darkblue}{rgb}{0,0,0.6}
\usepackage[ocgcolorlinks,colorlinks=true, citecolor=darkblue, filecolor=darkblue,
linkcolor=darkblue, urlcolor=darkblue]{hyperref}
\usepackage[nameinlink,capitalise,noabbrev]{cleveref}
\usepackage{bbm}
\usepackage{colonequals}  % this package is devoted to getting the := symbol to look as it should
\usepackage{enumerate}
\usepackage{tikz-cd}

\makeatletter
\@namedef{subjclassname@2010}{%
  \textup{2010} Mathematics Subject Classification}
\makeatother
\title[Shortening and K-theory of infinite products]{Shortening binary complexes and commutativity of K-theory with infinite products}
\author[D.~Kasprowski]{Daniel Kasprowski}
\address{Rheinische Friedrich-Wilhelms-Universit\"at Bonn, Mathematisches Institut,\newline En\-de\-nicher Allee 60, 53115 Bonn, Germany}
\email{kasprowski@uni-bonn.de}
\urladdr{http://www.math.uni-bonn.de/people/daniel/}

\author[C.~Winges]{Christoph Winges}
\address{Max-Planck-Institut f\"ur Mathematik, Vivatsgasse 7, 53111 Bonn, Germany}
\email{winges@mpim-bonn.mpg.de}
\urladdr{http://guests.mpim-bonn.mpg.de/winges/} 


\newcounter{commentcounter}

\date{\today}
\newcommand{\IN}{\mathbb{N}}

\newcommand{\bbQ}{\mathbbm{Q}}

\newcommand{\bbZ}{\mathbbm{Z}}
\newcommand{\bbN}{\mathbbm{N}}

\newcommand{\bbK}{\mathbb{K}}

\DeclareMathOperator{\id}{id}
\DeclareMathOperator{\supp}{supp}

\DeclareMathOperator{\Idem}{Idem}

\newcommand{\cN}{\mathcal{N}}
\newcommand{\cE}{\mathcal{E}}
\newcommand{\bJ}{\mathbb{J}}
\newcommand{\bP}{\mathbb{P}}
\newcommand{\bQ}{\mathbb{Q}}
\newcommand{\bS}{\mathbb{S}}
\newcommand{\bT}{\mathbb{T}}

\newcommand{\cF}{\mathcal{F}}
\newcommand{\cS}{\mathcal{S}}

\newcommand{\cA}{\mathcal{A}}

\newcommand{\comment}[1]{}

\renewcommand{\phi}{\varphi}

\newcommand{\N}{\mathbb{N}}

\newcommand{\bfW}{\mathbb{W}}

\newcommand{\odd}{\mathrm{odd}}
\newcommand{\even}{\mathrm{ev}}

\numberwithin{equation}{section}
\newtheorem{thm}[equation]{Theorem}
\newtheorem{theorem}[equation]{Theorem}
\newtheorem{prop}[equation]{Proposition}
\newtheorem{cor}[equation]{Corollary}
\newtheorem{lemma}[equation]{Lemma}
\theoremstyle{definition}

\newtheorem{definition}[equation]{Definition}
\newtheorem{rem}[equation]{Remark}
\newtheorem{remark}[equation]{Remark}

\begin{document}
\begin{abstract}
We show that in Grayson's model of higher algebraic K-theory using binary acyclic complexes, the complexes of length two suffice to generate the whole group. 
Moreover, we prove that the comparison map from Nenashev's model for $K_1$ to Grayson's model for $K_1$ is an isomorphism. It follows that algebraic $K$-theory of exact categories commutes with infinite products.
\end{abstract}
\maketitle
\section{Introduction}
On a conceptual level, the algebraic $K$-theory functor is by now well understood in terms of a universal property, which encapsulates the known fundamental properties of Quillen's or Waldhausen's construction \cite{Barwick16, BGT13}.

One of the more elusive properties of algebraic $K$-theory is its compatibility with infinite products.
This question was studied by Carlsson \cite{Carlsson95} in connection to work of Carlsson--Pedersen on the split injectivity of the $K$-theoretic assembly map \cite{CP95},
and permeates the literature adapting their ``descent'' argument to prove more general cases of the $K$-theoretic Novikov conjecture \cite{BR07, RTY14, Kas15}.
Carlsson's proof, while relying on the Additivity theorem, is for the most part concerned with simplicial techniques involving what he calls quasi-Kan complexes.

The present article aims to provide a different perspective on the question.
In \cite{Grayson2012}, Grayson showed that the higher algebraic $K$-theory of an exact category can be expressed in terms of binary acyclic complexes. See \cref{sec:complexes} for a quick review.

In \cite{Nenashev1998} Nenashev gave a different presentation $K^N_1(\cN)$ of $K_1(\cN)$ whose generators are binary acyclic complexes of length two.
Regarding a binary acyclic complex of length two as a class in $K_0(\Omega\cN)$ defines a natural homomorphism
\[ \Phi \colon K_1^N(\cN) \to K_0(\Omega\cN),\]
see \cite[Remark~8.1]{Grayson2012} and the beginning of \cref{sec:nenashev}.

Unpublished work of Grayson shows that $\Phi$ is a surjection, cf.~\cite[Remark~8.1]{Grayson2012}.
Building on Grayson's unpublished argument (see \cref{rem:grayson}), we improve this to a bijectivity statement.

\begin{theorem}\label{thm:nenashev_vs_grayson}
	The map $\Phi$ is an isomorphism.
\end{theorem}

We use this to show the following theorem.
\begin{thm}
	\label{thm:commute}
	For every family $\{\cN_i\}_{i\in I}$ of exact categories, the natural map
	\[\bbK^{-\infty}\left(\prod_{i\in I}\cN_i\right)\to \prod_{i\in I} \bbK^{-\infty}(\cN_i)\]
	is a $\pi_*$-isomorphism.
\end{thm}
Since Grayson's results in \cite{Grayson2012} rely only on the fundamental properties of $K$-theory,
our proof is not only elementary, but also exhibits \cref{thm:commute} as a consequence of the universal property of algebraic $K$-theory.

Since the proof of \cref{thm:nenashev_vs_grayson} is technical, we will begin by showing the following weaker statements. Here the proofs are considerably easier and they suffice to deduce \cref{thm:commute}.

In \cref{sec:shortening} we give a proof that the complexes of length four suffice to generate $K_0(\Omega \cN)$.
\begin{thm}
	\label{thm:dim4}
	The canonical map $K_0(\Omega_{[0,4]}\cN)\to K_0(\Omega\cN)$ is a surjection.
\end{thm}
A closer inspection of the constructions involved in proving \cref{thm:dim4} provides a candidate for a homomorphism $K_0(\Omega\cN) \to K_0(\Omega_{[0,4]}\cN)$.
Admitting slightly larger complexes, we show that this homomorphism is well-defined and provides a right inverse to the canonical comparison map.
\begin{thm}
	\label{thm:dim7}
	For every $n\in\N$ the canonical map $K_0(\Omega_{[0,7]}^n\cN) \to K_0(\Omega^n\cN) \cong K_n(\cN)$ admits a natural section.
\end{thm}
In \cref{sec:products}, we use this right inverse to show \cref{thm:commute}. Finally, we give the proof of \cref{thm:nenashev_vs_grayson} in \cref{sec:nenashev}.

\subsection*{Acknowledgments}
We are indebted to Daniel Grayson for sharing with us his proof of surjectivity of $\Phi$. We thank Robin Loose for helpful discussions. 

\section{Binary complexes}\label{sec:complexes}

In this section, we give a quick review of Grayson's description of the higher algebraic $K$-groups \cite{Grayson2012}.
In the following $\cN$ will always denote an exact category. Chain complexes in $\cN$ will always be assumed to be \emph{bounded}.
Denote by $C\cN$ the category of (bounded) chain complexes in $\cN$.

\begin{definition}
 A chain complex $(P_*,d)$ in $\cN$ is called \emph{acyclic} if each differential admits a factorization into an admissible epimorphism followed by an admissible monomorphism
 \[ d_n \colon P_n \twoheadrightarrow J_{n-1} \rightarrowtail P_{n-1} \]
 such that $J_n \rightarrowtail P_n \twoheadrightarrow J_{n-1}$ is a short exact sequence.
 
 Denote by $C^q\cN \subseteq C\cN$ the full subcategory of acyclic chain complexes.
\end{definition}

\begin{definition}
 A \emph{binary acyclic complex} $(P_*, d, d')$ is a graded object $P_*$ over $\cN$ together with two degree $-1$ maps $d, d' \colon P_* \to P_*$ such that both $(P_*,d)$ and $(P_*,d')$ are acyclic chain complexes.
 The differentials $d$ and $d'$ are called the \emph{top} and \emph{bottom} differential.
\end{definition}

A morphism of binary acyclic complexes is a degree $0$ map of underlying graded objects which is a chain map with respect to both differentials.
The resulting category of binary acyclic complexes is denoted by $B^q\cN$.
There is a natural exact functor $\Delta \colon C^q(\cN) \to B^q(\cN)$ which duplicates the differential of a given acyclic chain complex.

Fix $n > 0$.
Since both $C^q\cN$ and $B^q\cN$ are exact categories, these constructions can be iterated.
For any finite sequence $\bfW = (W_1,\dots,W_n)$ in $\{B,C\}$, denote by $\bfW^q\cN$ the category $W_1^q\dots W_n^q\cN$.
If $\bfW$ is the constant sequence on the letter $B$, we also write $(B^q)^n\cN$.
Letting $\bfW$ vary over all possible choices defines a commutative $n$-cube of exact categories
which induces a commutative $n$-cube of spectra upon taking algebraic $K$-theory.
The spectrum $\bbK(\Omega^n\cN)$ is defined to be the total homotopy cofiber of this cube.

We rely on the following result about $\bbK(\Omega^n\cN)$.

\begin{theorem}[{Grayson, \cite[Corollary~7.2]{Grayson2012}}]
 The abelian groups $K_n(\cN)$ and $K_0(\Omega^n\cN)$ are naturally isomorphic.
\end{theorem}

This theorem facilitates a completely algebraic description of higher $K$-theory \cite[Corollary~7.4]{Grayson2012}.
For example, it implies that $K_1(\cN)$ can be described as the Grothendieck group of the category of binary acyclic complexes $B^q\cN$
with the additional relation that a binary acyclic complex represents the trivial class if its top and bottom differential coincide.
We use this description of $K_1(\cN)$ extensively in \cref{sec:shortening}.

Throughout this article, we employ the following variations of this construction:
Let $J \subseteq \bbZ$ be an interval, i.e.~$J = \{ z \in \bbZ \mid a \leq z \leq b \}$ for some $a,b \in \bbZ \cup \{\pm\infty\}$.
Then we denote by $B^q_J\cN$ and $C^q_J\cN$ the categories of (binary) acyclic complexes supported on $J$.
Thus, any sequence of intervals $\bJ = (J_1,\dots,J_n)$ in $\bbZ$ gives rise to an abelian group $K_0(\Omega_\bJ\cN) := K_0(\Omega_{J_1}\dots\Omega_{J_n}\cN)$.
If $\bJ' = (J_1',\dots,J_n')$ is another such sequence satisfying $J_k \subseteq J_k'$ for all $k$, we have a natural homomorphism
\[ i_{\bJ,\bJ'} \colon K_0(\Omega_\bJ\cN) \to K_0(\Omega_{\bJ'}\cN). \]
Note that $\Delta \colon C^q_J\cN \to B^q_J\cN$ admits two natural splits $\top$ and $\bot$ which forget the bottom, respectively top, differential of a binary acyclic complex.
Using one of these, we see that $i_{\bJ,\bJ'}$ is naturally a retract of the homomorphism
\[ K_0(\Omega_{J_1}B^q_{J_2}\dots B^q_{J_n}\cN) \to K_0(\Omega_{J_1'}B^q_{J_2'}\dots B^q_{J_n'}\cN). \]
Moreover, we observe that any permutation $\sigma \colon \{1,\dots,n\} \xrightarrow{\cong} \{1,\dots,n\}$ induces an isomorphism
\[ K_0(\Omega_{I_1}\dots\Omega_{I_n}\cN) \cong K_0(\Omega_{I_\sigma(1)}\dots\Omega_{I_\sigma(n)}\cN). \]

It is notationally convenient to work with $\bbN$-graded bounded chain complexes instead of $\bbZ$-graded chain complexes.
The following lemma justifies this convention.

\begin{lemma}
 The natural map $K_0(\Omega_{[0,\infty)}^n\cN) \to K_0(\Omega^n\cN)$ is an isomorphism for all $n \geq 1$.
\end{lemma}
\begin{proof}
 We begin with the case $n=1$.
 The map $K_0(\Omega_{[0,\infty)}\cN) \to K_0(\Omega\cN)$ is an isomorphism since the class group of a filtered union is isomorphic to the colimit of the class groups and shifting induces an isomorphism in $K$-theory.
 
 We will now prove the lemma by induction. Assume that it holds for $n-1$.
 The map $K_0(\Omega^n_{[0,\infty)}\cN)\to K_0(\Omega\Omega^{n-1}_{[0,\infty)}\cN)$ is a retract of $K_0(\Omega_{[0,\infty)}(B^q_{[0,\infty)})^{n-1}\cN)\to K_0(\Omega(B^q_{[0,\infty)})^{n-1}\cN)$, which is an isomorphism by the induction beginning. Hence, $K_0(\Omega^n_{[0,\infty)}\cN)\to K_0(\Omega\Omega^{n-1}_{[0,\infty)}\cN)$ is an isomorphism as well. Using that $\Omega$ and $\Omega_{[0,\infty)}$ commute, it suffices to show that
 $K_0(\Omega^{n-1}_{[0,\infty)}\Omega\cN)\to K_0(\Omega^n\cN)$ is an isomorphism. This map is a retract of $K_0(\Omega^{n-1}_{[0,\infty)}B^q\cN)\to K_0(\Omega^{n-1}B^q\cN)$, which is an isomorphism by assumption.
\end{proof}

From now on, we write $K_0(\Omega\cN)$ for $K_0(\Omega_{[0,\infty)}\cN)$. All chain complexes considered in the sequel will be assumed to be positive.

In the remainder of this section, we record some important properties of $K_0(\Omega\cN)$.

\begin{definition}
 Let $\bP = (P_*, d, d')$ be a binary acyclic complex and let $i \in \IN$.
 \begin{enumerate}
  \item The \emph{$i$-th shift} $\bP[i]$ is defined to be the binary acyclic complex with underlying graded object $P[i]_* = P_{*-i}$ and differentials
   \[ \begin{tikzcd} P[i]_n = P_{n-i} \ar[r, shift right, "d_{n-i}'"']\ar[r, shift left, "d_{n-i}"] &[+10pt] P_{n-i-1} = P[i]_{n-1}\end{tikzcd} \]
  \item The \emph{$i$-th suspension} $\Sigma^i\bP$ is defined to be the binary acyclic complex with underlying graded object $\Sigma^iP_* = P_{*-i}$ and differentials
   \[ \begin{tikzcd} \Sigma^iP_n = P_{n-i} \ar[r, shift right, "(-1)^i d_{n-i}'"']\ar[r, shift left, "(-1)^i d_{n-i}"] &[+20pt] P_{n-i-1} = \Sigma^iP_{n-1}\end{tikzcd} \]
 \end{enumerate}
\end{definition}

\begin{remark}
 Our terminology is in disagreement with \cite{Grayson2012}, where the suspension is called a shift.
\end{remark}

As for ordinary chain complexes, we have the following lemma:

\begin{lemma}[{cf.~\cite[Lemma~6.1]{GraysonRelative} and \cite[Lemma~2.5]{HarrisPhD}}]\label{lem:shifting}
 Let $\bP$ be a binary acyclic complex. Then
 \[ [\bP[i]] = [\Sigma^i\bP] = (-1)^i[\bP] \in K_0(\Omega\cN). \]
\end{lemma}
\begin{proof}
 The first equality holds since $\bP[1] \cong \Sigma\bP$. The second equality holds since $\bP$ and $\Sigma\bP$ fit into a short exact sequence with the cone of $\bP$.
\end{proof}

\begin{definition}
 A \emph{binary double complex} is a bounded bigraded object $(P_{k,l})_{k,l \in \bbN}$ in $\cN$ together with morphisms
 \[ d^h_{k,l} \colon P_{k,l} \to P_{k-1,l}, \quad d^v_{k,l} \colon P_{k,l} \to P_{k,l-1} \]
 and
 \[ d^{\prime,h}_{k,l} \colon P_{k,l} \to P_{k-1,l}, \quad d^{\prime,v}_{k,l} \colon P_{k,l} \to P_{k,l-1} \]
 such that $(P_{*,*}, d^h, d^v)$ and $(P_{*,*}, d^{\prime,h}, d^{\prime,v})$ are double complexes in the sense that
 $(P_{*,l}, d^h)$ and $(P_{*,l},d^{\prime,h})$ are chain complexes for all $l$, $(P_{k,*},d^v)$ and $(P_{k,*},d^{\prime,v})$ are chain complexes for all $k$,
 and $d^hd^v =d^vd^h$, respectively $d^{\prime,h}d^{\prime,v} = d^{\prime,v}d^{\prime,h}$.
 
 We call $(P_{*,*}, d^h, d^v, d^{\prime,h}, d^{\prime,v})$ a \emph{binary acyclic double complex} if $(P_{*,l}, d^h, d^{\prime,h})$ is a binary acyclic complex for all $l$
 and $(P_{k,*}, d^v, d^{\prime,v})$ is a binary acyclic complex for all $k$.
\end{definition}

Let $(P_{*,*}, d^h, d^v, d^{\prime,h}, d^{\prime,v})$ be a binary acyclic double complex.
Forming the total complex of $(P_{*,*}, d^h, d^v)$ and $(P_{*,*}, d^{\prime,h}, d^{\prime,v})$, using the usual sign trick, produces a binary acyclic complex $\bT$.
Filtering $\bT$ according to the horizontal (respectively vertical) filtration of the double complexes and applying \cref{lem:shifting}
immediately gives the following lemma.

\begin{lemma}[{Nenashev's relation, cf.~\cite[Remark~8.1]{Grayson2012} and \cite[Proposition~2.10]{HarrisPhD}}]\label{lem:nena}
 Let $(P_{*,*}, d^h, d^v, d^{\prime,h}, d^{\prime,v})$ be a binary acyclic double complex.
 Then we have
 \[ \sum_l (-1)^l [P_{*,l}, d^h, d^{\prime,h}] = \sum_k (-1)^k [P_{k,*}, d^v, d^{\prime,v}] \]
 in $K_0(\Omega_{\supp \bT}\cN)$.
\end{lemma}

This relation is analogous to the relation used by Nenashev \cite{Nenashev1998} to define $K^N_1(\cN)$, hence its name.

\begin{remark}\label{rem:binarydouble_notation}
	Specifying a binary double complex involves a sizeable amount of data.
	In order to write down such complexes without occupying too much space, we will follow Nenashev's convention and depict binary double complexes by diagrams of the form
	\[\begin{tikzcd}
	\bullet\ar[r, shift right]\ar[r, shift left]\ar[d, shift right]\ar[d, shift left] & \bullet\ar[r, shift right]\ar[r, shift left]\ar[d, shift right]\ar[d, shift left] & \bullet\ar[d, shift right]\ar[d, shift left] \\
	\bullet\ar[r, shift right]\ar[r, shift left]\ar[d, shift right]\ar[d, shift left] & \bullet\ar[r, shift right]\ar[r, shift left]\ar[d, shift right]\ar[d, shift left] & \bullet\ar[d, shift right]\ar[d, shift left] \\
	\bullet\ar[r, shift right]\ar[r, shift left] & \bullet\ar[r, shift right]\ar[r, shift left] & \bullet
	\end{tikzcd} \]
	where it is understood that the left vertical morphisms commute with the top horizontal morphisms (corresponding to $d^h$ and $d^v$),
	and that the right vertical morphisms commute with the bottom horizontal morphisms (corresponding to $d^{\prime,h}$ and $d^{\prime,v}$).
\end{remark}

\begin{lemma}\label{lem:ordertwo}
 Let $J$ be an object in $\cN$ and denote by
 \[ \tau_J := \begin{pmatrix} 0 & \id_J \\ \id_J & 0 \end{pmatrix} \colon J \oplus J \to J \oplus J \]
 the automorphism which switches the two summands.
 Then the element
 \[ [ \begin{tikzcd} J\oplus J\ar[r, shift left, "\id"]\ar[r, shift right, "\tau_J"']&[-5pt] J\oplus J\end{tikzcd} ] \in K_0(\Omega_{[0,2]}\cN) \]
 has order two.
\end{lemma}
\begin{proof}
 An application of \cref{lem:nena} to the binary acyclic double complex
 \[\begin{tikzcd}
   J \oplus J\ar[r, shift left, "\tau_J"]\ar[r, shift right, "\id_{J\oplus J}"']\ar[d, shift left, "\tau_J"]\ar[d, shift right, "\tau_J"'] & J \oplus J\ar[d, shift left, "\id_{J\oplus J}"]\ar[d, shift right, "\tau_J"'] \\
   J \oplus J\ar[r, shift left, "\tau_J"]\ar[r, shift right, "\tau_J"'] & J \oplus J
  \end{tikzcd}\]
 shows that
 \[
 [\begin{tikzcd}
  J\oplus J\ar[r, shift left, "\id_{J\oplus J}"]\ar[r, shift right, "\tau_J"']&[-5pt] J\oplus J\end{tikzcd}]=-[\begin{tikzcd}J\oplus J\ar[r, shift left, "\tau_J"]\ar[r, shift right, "\id_{J\oplus J}"']&[-5pt] J\oplus J
  \end{tikzcd}]\]
 in $K_0(\Omega_{[0,2]}\cN)$; cf.\ also \cite[Corollary~8.7]{GraysonRelative}.
 On the other hand, the isomorphism
  \[\begin{tikzcd}
  J \oplus J\ar[r, shift left, "\tau_J"]\ar[r, shift right, "\id_{J\oplus J}"']\ar[d, "\tau_J"'] & J \oplus J\ar[d, "\id_{J\oplus J}"] \\
  J \oplus J\ar[r, shift left, "\id_{J\oplus J}"]\ar[r, shift right, "\tau_J"'] & J \oplus J
  \end{tikzcd}\]
  of binary acyclic complexes implies
  \[[\begin{tikzcd}
  	J\oplus J\ar[r, shift left, "\id_{J\oplus J}"]\ar[r, shift right, "\tau_J"']&[-5pt] J\oplus J\end{tikzcd}]=[\begin{tikzcd}J\oplus J\ar[r, shift left, "\tau_J"]\ar[r, shift right, "\id_{J\oplus J}"']&[-5pt] J\oplus J
  \end{tikzcd}]\in K_0(\Omega_{[0,2]}\cN).\qedhere\]
\end{proof}

\section{Shortening binary complexes}\label{sec:shortening}

The goal of this section is to prove \cref{thm:dim4} and \cref{thm:dim7}. As before, $\cN$ denotes an exact category.
The basic approach is the same as that of Harris \cite[Section~2.2]{HarrisPhD} in showing that the canonical map from Bass' $K_1$ to $K_0(\Omega\cN)$ is an isomorphism for split-exact categories.
Our arguments rely on a description of equality of classes in $K_0$ of an exact category which is due to Heller \cite[Lemma~2.1]{Heller1965}.
We include a proof following \cite[Lemma 2.4]{thomason} for the reader's convenience.

\begin{definition}
 Let $J,K \in \cN$.
 \begin{enumerate}
  \item We call $J$ and $K$ \emph{extension-equivalent} if there are objects $A, B \in \cN$ such that there exist exact sequences
    \[
     \begin{tikzcd}
      A\ar[r, rightarrowtail] & J\ar[r, twoheadrightarrow] & B
     \end{tikzcd}
    \quad\text{and}\quad
    \begin{tikzcd}
     A\ar[r, rightarrowtail] & K\ar[r, twoheadrightarrow] & B.
    \end{tikzcd}
 \]
 \item We call $J$ and $K$ \emph{stably extension-equivalent} if there exists an object $S \in \cN$ such that $J \oplus S$ and $K \oplus S$ are extension-equivalent.
 \end{enumerate}
\end{definition}

Despite its name, extension-equivalence need not be an equivalence relation.
On the other hand, the following lemma shows that stable extension-equivalence is always an equivalence relation.

\begin{lemma}[Heller]\label{lem:hellerscriterion}
 Let $\cN$ be an exact category and let $J, J', K, K' \in \cN$.
 
 Then $[J] - [J'] = [K] - [K'] \in K_0(\cN)$ if and only if $J \oplus K'$ and $K \oplus J'$ are stably extension-equivalent.
\end{lemma}
\begin{proof}
 Define a relation on pairs of objects in $\cN$ by setting $(J,J') \sim (K,K')$ if and only if $J \oplus K'$ and $K \oplus J'$ are stably extension-equivalent.
 
 We claim that $\sim$ is an equivalence relation.
 Reflexivity and symmetry are obvious.
 To see transitivity, suppose that $(J,J') \sim (K,K') \sim (L,L')$, i.e.\ there exist $A$, $B$, $C$, $D$, $S$, $T \in \cN$ such that there are exact sequences
 \[
   \begin{tikzcd}
	A\ar[r, rightarrowtail] & J \oplus K' \oplus S\ar[r, twoheadrightarrow] & B
	\end{tikzcd}
   \quad\text{and}\quad
    \begin{tikzcd}
	A\ar[r, rightarrowtail] & K \oplus J' \oplus S\ar[r, twoheadrightarrow] & B
	\end{tikzcd} \]
 as well as
 \[
   \begin{tikzcd}
	C\ar[r, rightarrowtail] & K \oplus L' \oplus T\ar[r, twoheadrightarrow] & D
	\end{tikzcd}
   \quad\text{and}\quad
    \begin{tikzcd}
	C\ar[r, rightarrowtail] & L \oplus K' \oplus T\ar[r, twoheadrightarrow] & D.
	\end{tikzcd}
 \]
 Then the sequences formed by taking direct sums
 \[
   \begin{tikzcd}
	A \oplus C\ar[r, rightarrowtail] & J \oplus K' \oplus S \oplus K \oplus L' \oplus T\ar[r, twoheadrightarrow] & B \oplus D,
	\end{tikzcd}
 \]
 \[
   \begin{tikzcd}
	A \oplus C\ar[r, rightarrowtail] & K \oplus J' \oplus S \oplus L \oplus K' \oplus T\ar[r, twoheadrightarrow] & B \oplus D
   \end{tikzcd}
 \]
 are exact, too. Rewriting
 \[ J \oplus K' \oplus S \oplus K \oplus L' \oplus T \cong J \oplus L' \oplus K \oplus K' \oplus S \oplus T \]
 and
 \[ K \oplus J' \oplus S \oplus L \oplus K' \oplus T \cong L \oplus J' \oplus K \oplus K' \oplus S \oplus T \]
 proves transitivity, so $\sim$ is an equivalence relation.
 Denote by $k(\cN)$ the set of equivalence classes in $ob\cN \times ob\cN$ with respect to $\sim$.
 We write $[J,J']$ for the class of $(J,J')$ in $k(\cN)$.
 
 Clearly, if $(J,J')$ and $(K,K')$ are pairs of objects such that $J \cong K$ and $J' \cong K'$, then $[J,J'] = [K,K']$.
 Hence, the direct sum operation in $\cN$ induces the structure of a commutative monoid on $k(\cN)$ via
 \[ [J,J'] + [K,K'] := [J \oplus K, J' \oplus K']. \]
 It is easy to check that $[J,J] = [0,0]$ for every object $J \in \cN$, so $k(\cN)$ is an abelian group since
 \[ [J,J'] + [J',J] = [J \oplus J', J \oplus J'] = 0. \]
 Let now $J \rightarrowtail K \twoheadrightarrow L$ be an exact sequence in $\cN$. Since both
 \[
  \begin{tikzcd}
	J\ar[r, rightarrowtail] & J \oplus L\ar[r, twoheadrightarrow] & L.
  \end{tikzcd}
 \quad\text{and}\quad
  \begin{tikzcd}
	J \ar[r, rightarrowtail] & K\ar[r, twoheadrightarrow] & L
  \end{tikzcd}
 \]
 are exact, it follows that $[J \oplus L, 0] = [K,0]$.
 Hence, the map $ob\cN \to k(\cN), J \mapsto [J,0]$ induces a homomorphism $\phi \colon K_0(\cN) \to k(\cN)$.
 
 Note that $\phi$ sends the class $[J] - [J'] \in K_0(\cN)$ to $\phi([J] - [J']) = [J,J']$, so $\phi$ is an epimorphism.
 Moreover, it is immediate from the definition of $\sim$ that the kernel of $\phi$ is trivial.
 This proves that $\phi$ is an isomorphism, and the claim of the lemma follows.
\end{proof}

We can now prove \cref{thm:dim4}. Let $\bP:=(P_*, d, d')$ be a binary acyclic complex supported on $[0,m]$ for some $m\in\bbN$.
Choose factorizations $d_n \colon P_n \twoheadrightarrow J_{n-1} \rightarrowtail P_{n-1}$ and $d_n' \colon P_n \twoheadrightarrow K_{n-1} \rightarrowtail P_{n-1}$ for all $n$.
Since $J_n$ and $K_n$ both fit into an exact sequence with $P_{n-1},\ldots, P_0$, they represent the same class in $K_0(\cN)$.
Therefore, there exist $A_n,B_n,S_n\in\cN$ and exact sequences 
\[\begin{tikzcd}
 A_n\ar[r, rightarrowtail] & J_n \oplus S_n\ar[r, twoheadrightarrow] & B_n
 \end{tikzcd}
 \quad\text{and}\quad
 \begin{tikzcd}
 A_n\ar[r, rightarrowtail] & K_n \oplus S_n\ar[r, twoheadrightarrow] & B_n.
 \end{tikzcd}\]
 For $n \geq 3$, let $\bS_n$ denote the binary acyclic complex
 \[\begin{tikzcd}[column sep = small]
 A_n\ar[r, shift left]\ar[r, shift right] & K_n \oplus S_n \oplus J_n\ar[r, shift left]\ar[r, shift right]&B_n\oplus P_n\oplus A_{n-1}\ar[r, shift left]\ar[r, shift right]&J_{n-1}\oplus K_{n-1} \oplus S_{n-1}\ar[r, shift left]\ar[r, shift right]&B_{n-1}
 \end{tikzcd}\]
 consisting of top differential
 \[\begin{tikzcd}[row sep= tiny]
  A_n\ar[r, rightarrowtail] & K_n \oplus S_n\ar[r, twoheadrightarrow]\ar[d, phantom, "\oplus"] & B_n\ar[d, phantom, "\oplus"] & & \\
  & J_n\ar[r, rightarrowtail] & P_n\ar[r, twoheadrightarrow, "d_n"]\ar[d, phantom, "\oplus"] & J_{n-1}\ar[d, phantom, "\oplus"] & \\
  & & A_{n-1}\ar[r, rightarrowtail] & K_{n-1}\oplus S_{n-1}\ar[r, twoheadrightarrow] & B_{n-1}
\end{tikzcd}\]
 and bottom differential
 \[\begin{tikzcd}[row sep= tiny]
  A_n\ar[r, rightarrowtail] & J_n\oplus S_n\ar[r, twoheadrightarrow]\ar[d, phantom, "\oplus"]& B_n\ar[d, phantom, "\oplus"]&&\\
  & K_n\ar[r, rightarrowtail] & P_n\ar[r, twoheadrightarrow, "d_n'"]\ar[d, phantom, "\oplus"]& K_{n-1}\ar[d, phantom, "\oplus"]&\\
  & & A_{n-1}\ar[r, rightarrowtail] & J_{n-1} \oplus S_{n-1}\ar[r, twoheadrightarrow]& B_{n-1}.
  \end{tikzcd}\] 
 Note that $\bS_n$ is zero for almost all $n$.
Furthermore, let $\bS_2$ denote the binary acyclic complex
\[\begin{tikzcd}
A_2\ar[r, shift left]\ar[r, shift right]&K_2\oplus S_2\oplus J_2\ar[r, shift left]\ar[r, shift right]&B_2\oplus P_2\ar[r, shift left]\ar[r, shift right]&P_1\ar[r, shift left]\ar[r, shift right]&P_0
\end{tikzcd}\]
consisting of top differential
\[\begin{tikzcd}[row sep= tiny]
 A_2\ar[r, rightarrowtail]&K_2\oplus S_2\ar[r, twoheadrightarrow]\ar[d, phantom, "\oplus"]& B_2\ar[d, phantom, "\oplus"]&&\\
 &J_2\ar[r, rightarrowtail]&P_2\ar[r, "d_2"]& P_1\ar[r, twoheadrightarrow, "d_1"]&P_0
\end{tikzcd}\]
and bottom differential
\[\begin{tikzcd}[row sep= tiny]
 A_2\ar[r, rightarrowtail]&J_2\oplus S_2\ar[r, twoheadrightarrow]\ar[d, phantom, "\oplus"]& B_2\ar[d, phantom, "\oplus"]&&\\
 &K_2\ar[r, rightarrowtail]&P_2\ar[r, "d_2'"]& P_1\ar[r, twoheadrightarrow, "d_1'"]&P_0.
\end{tikzcd}\]

\begin{lemma}
	\label{lem:dim4}
	\[[\bP]=\sum_{n=2}^{\infty}(-1)^n[\bS_n]\in K_0(\Omega_{[0,m+3]} \cN)\]
\end{lemma}
\begin{proof}
	Let $\bP'$ denote the binary acyclic complex
	\[\begin{tikzcd}
	\ldots\ar[r, shift left]\ar[r, shift right]&P_4\ar[r, shift left]\ar[r, shift right]&P_3\oplus A_2\ar[r, shift left]\ar[r, shift right]&J_2\oplus K_2\oplus S_2\ar[r, shift left]\ar[r, shift right]&B_2
	\end{tikzcd}\]
	with top differential
	\[\begin{tikzcd}[row sep= tiny]
	\ldots\ar[r]&P_4\ar[r, "d_4"]&P_{3}\ar[r, twoheadrightarrow, "d_{3}"]\ar[d, phantom, "\oplus"]&J_2\ar[d, phantom, "\oplus"]&\\
	&&A_{2}\ar[r, rightarrowtail]&K_{2}\oplus S_2\ar[r, twoheadrightarrow]& B_{2}
	\end{tikzcd}\]
	and bottom differential
		\[\begin{tikzcd}[row sep= tiny]
		\ldots\ar[r]&P_4\ar[r, "d'_4"]&P_{3}\ar[r, twoheadrightarrow, "d'_{3}"]\ar[d, phantom, "\oplus"]&K_2\ar[d, phantom, "\oplus"]&\\
		&&A_{2}\ar[r, rightarrowtail]&J_{2}\oplus S_2\ar[r, twoheadrightarrow]& B_{2}
		\end{tikzcd}\]
	We will show that $[ \bP]=[\bS_2]-[\bP']$. The lemma then follows by iterating this procedure.
	Consider the following binary acyclic double complex. All differentials written as a single arrow are the identity on the summand appearing in domain and codomain and zero on all other summands. In particular, both differentials agree in this case. The remaining four non-trivial binary acyclic complexes are $\bP, \bP', \bS_2$ and a fourth one explained in the diagram.
	\[\begin{tikzcd}
	 &&A_2\ar[d]\ar[r]&A_2\ar[d, shift left]\ar[d, shift right]&&[+5pt]\\
	 \ldots\ar[r, shift left]\ar[r, shift right]&P_4\ar[r, shift left]\ar[r, shift right]\ar[d]&P_3\oplus A_2\ar[r, shift left]\ar[r, shift right]\ar[d]&J_2\oplus K_2\oplus S_2\ar[r, shift left]\ar[r, shift right]\ar[d, shift left]\ar[d, shift right]&B_2\ar[d]&\\
	 \ldots\ar[r, shift left]\ar[r, shift right]&P_4\ar[r, shift left]\ar[r, shift right]&P_3\ar[r, shift left]\ar[r, shift right]&P_2\oplus B_2\ar[r, shift left]\ar[r, shift right]\ar[d, shift left]\ar[d, shift right]&P_1\oplus B_2\ar[r, shift left]\ar[r, shift right]\ar[d]&P_0\ar[d]\\
	 &&&P_1\ar[r, shift left, "{(\id, d_1)}"]\ar[r, shift right, "{(\id, d_1')}"'] \ar[d, shift left]\ar[d, shift right]&P_1\oplus P_0\ar[r, shift left, "d_1-\id"]\ar[r, shift right, "d_1'-\id"']\ar[d]&P_0\\
	 &&&P_0 \ar[r]&P_0&
	\end{tikzcd}\]
	Applying Nenashev's relation (\cref{lem:nena}) and omitting all summands which are obviously zero, we obtain
	\[ -[\begin{tikzcd} P_1\ar[r, shift left, "{(\id, d_1)}"]\ar[r, shift right, "{(\id, d_1')}"'] &P_1 \oplus P_0\ar[r, shift left, "d_1-\id"]\ar[r, shift right, "d_1'-\id"'] & P_0 \end{tikzcd}] + [\bP] - [\bP'[1]] = [\bS_2]. \]
	We will show that the first summand is trivial. Assuming this, it follows from \cref{lem:shifting} that
	\[ [\bP] + [\bP'] = [\bS_2] \]
	as claimed.
	In fact, triviality of the binary acyclic complex
	\[\begin{tikzcd}
	 P_1\ar[r, shift left, "{(\id, d_1)}"]\ar[r, shift right, "{(\id, d_1')}"'] &[+5pt] P_1\oplus P_0\ar[r, shift left, "d_1-\id"]\ar[r, shift right, "d_1'-\id"'] &[+5pt] P_0
	\end{tikzcd} \]
	in $K_0(\Omega\cN)$ follows directly from the existence of the following short exact sequence of binary acyclic complexes:
	\[\begin{tikzcd}[baseline=(current bounding box.south)]
	 &[+5pt] P_0\ar[d, rightarrowtail]\ar[r, shift left, "-\id"]\ar[r, shift right, "-\id"'] &[+5pt] P_0\ar[d, rightarrowtail]\\
	 P_1\ar[r, shift left, "{(\id, d_1)}"]\ar[r, shift right, "{(\id, d_1')}"'] \ar[d, twoheadrightarrow]&P_1\oplus P_0\ar[r, shift left, "d_1-\id"]\ar[r, shift right, "d_1'-\id"']\ar[d, twoheadrightarrow]&P_0\\
	 P_1\ar[r]&P_1&
	\end{tikzcd}\qedhere\]
\end{proof}

\cref{lem:dim4} immediately implies \cref{thm:dim4} since the complexes $\bS_n$ are supported on $[0,4]$ for all $n \geq 2$.

Our next goal is to prove \cref{thm:dim7}.
\begin{prop}
 \label{prop:split}
 The map $K_0\Omega \cN\to K_0\Omega_{[0,7]}\cN$ given by
 \[ [\bP]\mapsto \sum_{n=2}^{\infty}(-1)^n[\bS_n] \]
 is a well-defined homomorphism.
\end{prop}
\begin{proof}
	 Note that all $J_n$ and $K_n$ are unique up to isomorphism.
	
	We first show that $\sum_{n=2}^{\infty}(-1)^n[\bS_n]$ is independent of the choices of $A_n, B_n, S_n$ and the extensions
	\[\begin{tikzcd}
	A_n\ar[r, rightarrowtail]&J_n\oplus S_n\ar[r, twoheadrightarrow]& B_n
	\end{tikzcd} \quad\text{and}\quad\begin{tikzcd}
	A_n\ar[r, rightarrowtail]&K_n\oplus S_n\ar[r, twoheadrightarrow]& B_n.
	\end{tikzcd}\]
	Fix $k > 2$, and let $A_k'$, $B_k'$ and $S_k'$ be different choices fitting into extensions as $A_k$, $B_k$ and $S_k$.
	Denote by $\bS_k'$ and $\bS_{k+1}'$ the same binary acyclic complexes as $\bS_k$ and $\bS_{k+1}$, except that the extensions involving $A_k$, $B_k$ and $S_k$ are replaced by those involving $A_k'$, $B_k'$ and $S_k'$. Note that $\bS_l$ is independent of the choice of $A_k$, $B_k$ and $S_k$ for $l \neq k, k+1$.
	
	The binary acyclic complexes $\bS_k' \oplus (\bS_{k+1}[2])$ and $\bS_k \oplus (\bS_{k+1}'[2])$ have isomorphic underlying graded objects.
	We regard both as binary acyclic complexes
	 {\footnotesize\[\begin{tikzcd}
	  A_{k+1}\ar[r, shift left]\ar[r, shift right] &[-10pt] K_{k+1} \oplus S_{k+1} \oplus J_{k+1}\ar[r, shift left]\ar[r, shift right] &[-10pt] B_{k+1} \oplus P_{k+1} \oplus A_k \oplus A_k'\ar[r, shift left]\ar[r, shift right]&[-10pt]~\end{tikzcd}\]\[\begin{tikzcd}
	J_k \oplus K_k \oplus S_k \oplus K_k \oplus S_k' \oplus J_k\ar[r, shift left]\ar[r, shift right] &[-10pt] B_k \oplus B_k' \oplus P_k \oplus A_{k-1} \ar[r, shift left]\ar[r, shift right] &[-10pt] J_{k-1} \oplus K_{k-1} \oplus S_{k-1}\ar[r, shift left]\ar[r, shift right] &[-10pt] B_{k-1},
	 \end{tikzcd}\]}
	Both the pair of chain complexes given by the top differentials and the pair of chain complexes given by the bottom differentials of $\bS_k' \oplus (\bS_{k+1}[2])$ and $\bS_k \oplus (\bS_{k+1}'[2])$ are isomorphic:
	The isomorphism for the top differentials has to flip the two copies of $K_k$, while the one for the bottom differentials has to flip the two copies of $J_k$.
	
	That is, there is the following binary acyclic double complex whose upper row is $\bS'_k \oplus (\bS_{k+1}[2])$, whose lower row is $\bS_k \oplus (\bS_{k+1}'[2])$. Here all unmarked downward arrows are the identity, and $\tau_K$ and $\tau_J$ denote the automorphisms switching the two copies of $K_k$ and $J_k$, respectively.
	 {\footnotesize\[\begin{tikzcd}
	  \ldots\ar[r, shift left]\ar[r, shift right] &[-15pt] B_{k+1} \oplus P_{k+1} \oplus A_k \oplus A_k'\ar[r, shift left]\ar[r, shift right]\ar[d] &[-15pt] J_k \oplus K_k \oplus S_k \oplus K_k \oplus S_k' \oplus J_k\ar[r, shift left]\ar[r, shift right]\ar[d, shift left, "\tau_J"]\ar[d, shift right, "\tau_K"'] &[-15pt] B_k \oplus B_k' \oplus P_k \oplus A_{k-1}\ar[d]\ar[r, shift left]\ar[r, shift right] &[-15pt] \ldots\\
	  \ldots\ar[r, shift left]\ar[r, shift right] & B_{k+1} \oplus P_{k+1} \oplus A_k \oplus A_k'\ar[r, shift left]\ar[r, shift right] & J_k \oplus K_k \oplus S_k \oplus K_k \oplus S_k' \oplus J_k \ar[r, shift left]\ar[r, shift right] & B_k \oplus B_k' \oplus P_k \oplus A_{k-1}\ar[r, shift left]\ar[r, shift right] & \ldots
	 \end{tikzcd}\]	}
	 Applying \cref{lem:nena}, the difference between the classes of $\bS_k\oplus (\bS'_{k+1}[2])$ and $\bS'_k\oplus (\bS_{k+1}[2])$ is therefore the same as
	 \[
	 [\begin{tikzcd}
	 K_k\oplus K_k\ar[r, shift left, "\tau_K"]\ar[r, shift right, "\id"'] &[-5pt] K_k\oplus K_k\end{tikzcd}]+
	 [\begin{tikzcd}
		J_k\oplus J_k\ar[r, shift left, "\id"]\ar[r, shift right, "\tau_J"']&[-5pt] J_k\oplus J_k
	 \end{tikzcd}]\]
	 in $K_0(\Omega_{[0,7]}\cN)$.
	 Since $J_k$ and $K_k$ represent the same class in $K_0$, we have
	 \[ [\begin{tikzcd} K_k\oplus K_k\ar[r, shift left, "\tau_K"]\ar[r, shift right, "\id"']&[-5pt] K_k\oplus K_k\end{tikzcd}]=[\begin{tikzcd}J_k\oplus J_k\ar[r, shift left, "\tau_J"]\ar[r, shift right, "\id"']&[-5pt]J_k\oplus J_k \end{tikzcd}]. \]
	 Therefore,  $\bS'_k\oplus (\bS_{k+1}[2])$ and $\bS_k\oplus (\bS'_{k+1}[2])$ represent the same class in $K_0(\Omega_{[0,7]}\cN)$ by \cref{lem:ordertwo}.
	 In combination with \cref{lem:shifting}, this shows
	 \[ [\bS_k]-[\bS_{k+1}]=[\bS'_k]-[\bS'_{k+1}]. \]
	 An analogous argument works for $k=2$, so the class $\sum_{n=2}^{\infty}(-1)^n[\bS_n]$ is independent of the choices we make.
	 
	 Next, we show that the map is independent of the choice of the representative $\bP$ of $[\bP]$. First note that if both differentials of the double complex $\bP$ agree, then $K_n$ and $J_n$ agree and we can choose the same extension for both. In this case, both differentials for all $\bS_n$ agree, so $[\bS_n] = 0$ for all $n$.
	 
	 It remains to see that for a short exact sequence $\bP'\rightarrowtail \bP\twoheadrightarrow\bP''$ we also get short exact sequences $\bS'_n\rightarrowtail \bS_n\twoheadrightarrow \bS''_n$ for all $n \geq 2$.
	 For every $n$, we have short exact sequences $J'_n\rightarrowtail J_n\twoheadrightarrow J_n''$ and $K'_n\rightarrowtail K_n\twoheadrightarrow K''_n$.
	 As above, the $K_0$-classes of $K'_n$ and $J'_n$ as well as those of $K_n''$ and $J_n''$ agree.
	 By the Additivity theorem \cite[Theorem~2]{Quillen1973}, we have
	 \[ [ K'_n\rightarrowtail K_n\twoheadrightarrow K''_n ] = [ J'_n\rightarrowtail J_n\twoheadrightarrow J_n'' ] \in K_0(\cE\cN), \]
	 where $\cE\cN$ is the exact category of exact sequences in $\cN$.
	 Therefore, we find short exact sequences $A'_n \rightarrowtail A_n\twoheadrightarrow A_n''$, $B'_n \rightarrowtail B_n \twoheadrightarrow B_n''$ and $S'_n \rightarrowtail S_n \twoheadrightarrow S''_n$ fitting into short exact sequences of short exact sequences:
	 \[\begin{tikzcd}
	 A_n'\ar[r, rightarrowtail]\ar[d, rightarrowtail]&J'_n\oplus S'_n\ar[r, twoheadrightarrow]\ar[d, rightarrowtail]&B_n'\ar[d, rightarrowtail]\\
	 A_n\ar[r, rightarrowtail]\ar[d, twoheadrightarrow]&J_n \oplus S_n\ar[r, twoheadrightarrow]\ar[d, twoheadrightarrow]&B_n\ar[d, twoheadrightarrow]\\
	 A_n''\ar[r, rightarrowtail]&J_n''\oplus S''_n\ar[r, twoheadrightarrow]&B_n''
	 \end{tikzcd}
	 \quad\text{and}\quad
	 \begin{tikzcd}
	 	A_n'\ar[r, rightarrowtail]\ar[d, rightarrowtail]&K'_n\oplus S'_n\ar[r, twoheadrightarrow]\ar[d, rightarrowtail]&B_n'\ar[d, rightarrowtail]\\
	 	A_n\ar[r, rightarrowtail]\ar[d, twoheadrightarrow]&K_n\oplus S_n \ar[r, twoheadrightarrow]\ar[d, twoheadrightarrow]&B_n\ar[d, twoheadrightarrow]\\
	 	A_n''\ar[r, rightarrowtail]&K_n''\oplus S''_n\ar[r, twoheadrightarrow]&B_n''
	 \end{tikzcd}
	 \]
	 Note that the middle vertical exact sequences are direct sums of the given sequences.
	 Using these extensions for the definition of $\bS'_n,\bS_n$ and $\bS''_n$, we get the desired short exact sequence $\bS'_n\rightarrowtail \bS_n\twoheadrightarrow \bS''_n$.
\end{proof}

\begin{proof}[Proof of \cref{thm:dim7}]
\cref{lem:dim4} and \cref{prop:split} prove the case $n=1$. The case $n > 1$ follows by induction, compare \cite[Remark~8.1]{Grayson2012}.

The map $K_0(\Omega^n_{[0,7]}\cN)\to K_0(\Omega\Omega_{[0,7]}^{n-1}\cN)$ is a retract of $K_0(\Omega_{[0,7]}(B^q_{[0,7]})^{n-1}\cN)\to K_0(\Omega(B^q_{[0,7]})^{n-1}\cN)$ which admits a natural section by the case $n=1$. Hence $K_0(\Omega^n_{[0,7]}\cN)\to K_0(\Omega\Omega_{[0,7]}^{n-1}\cN)$ admits a natural section as well.

Since $\Omega_{[0,7]}$ and $\Omega$ commute, it suffices to show that $K_0(\Omega_{[0,7]}^{n-1}\Omega\cN)\to K_0(\Omega^n\cN)$ admits a natural section. But this map is a retract of the map $K_0(\Omega_{[0,7]}^{n-1}B^q\cN)\to K_0(\Omega^{n-1}B^q\cN)$, which admits a natural section by the induction assumption.
\end{proof}

\section{Algebraic $K$-theory of infinite product categories}\label{sec:products}

The results of \cref{sec:shortening} allow us to show that the comparison map $\bbK(\prod_{i\in I}\cN_i) \to \prod_{i\in I} \bbK(\cN_i)$ of connective $K$-theory spectra is a $\pi_*$-isomorphism:

\begin{thm}
\label{thm:commute1}
For every family $\{\cN_i\}_{i\in I}$ of exact categories and every $n\in\bbN$ the natural map
\[ K_n(\prod_{i\in I}\cN_i) \to \prod_{i\in I} K_n(\cN_i)\]
is an isomorphism.
\end{thm}
\begin{proof}
 Note that the natural map $K_0(\prod_{i\in I}\cN_i) \to \prod_{i\in I} K_0(\cN_i)$ is clearly surjective, and that injectivity is a consequence of \cref{lem:hellerscriterion}.	
	
 Recall that $K_n(\cN)$ is naturally isomorphic to $K_0(\Omega^n\cN)$.
 Consider the following diagram, where the vertical maps are the sections from \cref{thm:dim7} followed by the canonical homomorphisms. 
 \[\begin{tikzcd}
  K_0(\Omega^n\prod_{i\in I}\cN_i)\ar[r]\ar[d] & \prod_{i\in I} K_0(\Omega^n\cN_i)\ar[d]\\
  K_0(\Omega^n_{[0,7]}\prod_{i\in I}\cN_i) \ar[d]\ar[r, "\cong"]&\prod_{i\in I} K_0(\Omega^n_{[0,7]}\cN_i) \ar[d]\\
  K_0(\Omega^n\prod_{i\in I}\cN_i)\ar[r] & \prod_{i\in I} K_0(\Omega^n\cN_i)
 \end{tikzcd}\]
 Since the natural functors $C^q_{[0,7]}(\prod_{i\in I}\cN_i) \to \prod_{i\in I} C^q_{[0,7]}(\cN_i)$ and $B^q_{[0,7]}(\prod_{i\in I}\cN_i) \to \prod_{i\in I} B^q_{[0,7]}(\cN_i)$ are isomorphisms,
 the middle horizontal map is an isomorphism.
 A diagram chase implies that the natural map $K_n(\prod_{i\in I}\cN_i) \to \prod_{i\in I} K_n(\cN_i)$ is an isomorphism.
\end{proof}

In the remainder of this section, we extend this statement to non-connective $K$-theory.
Our model for the non-connective algebraic $K$-theory $\bbK^{-\infty}$ of an exact category is Schlichting's delooping \cite[Section~12]{Schlichting2006}.

The argument to extend \cref{thm:commute1} to non-connective algebraic $K$-theory is based on a localization sequence of Schlichting \cite{Schlichting2004}.
To state it, we need to recall the following definition.

\begin{definition}[{\cite[Definition~1.3 and 1.5]{Schlichting2004}}]
 Let $\cN$ be an exact category, and let $\cA \subseteq \cN$ be an extension closed full subcategory.
 \begin{enumerate}
  \item An admissible epimorphism $N \twoheadrightarrow A$ with $N \in \cN$ and $A \in \cA$ is \emph{special} if there exists an admissible monomorphism $B \rightarrowtail N$ with $B \in \cA$ such that the composition $B \to A$ is an admissible epimorphism.
  \item The inclusion $\cA \subseteq \cN$ is called \emph{left $s$-filtering} if the following holds:
   \begin{enumerate}
    \item The subcategory $\cA$ is closed under admissible subobjects and admissible quotients in $\cN$.
    \item Every admissible epimorphism $N \twoheadrightarrow A$ from an object $N \in \cN$ to an object $A \in \cA$ is special.
    \item For every morphism $f \colon A \to N$ with $A \in \cA$ and $N \in \cN$ there exists an object $B \in \cA$, a morphism $f' \colon A \to B$ and an admissible monomorphism $i \colon B \rightarrowtail N$ such that $if' = f$.
   \end{enumerate}
 \end{enumerate}
\end{definition}

Let $\cA \subseteq \cN$ be a left $s$-filtering subcategory.
A \emph{weak isomorphism} in $\cN$ is a morphism which can be written as the composition of admissible monomorphisms with cokernel in $\cA$ and admissible epimorphisms with kernel in $\cA$.
Let $\Sigma$ denote the collection of weak isomorphisms in $\cN$.
The set $\Sigma$ satisfies a calculus of left fractions \cite[Lemma~1.13]{Schlichting2004}, so one can form the localization $\cN[\Sigma^{-1}]$.
The localization inherits an exact structure from $\cN$ by declaring a sequence to be exact if it is isomorphic to the image of an exact sequence under the localization functor $\cN \to \cN[\Sigma^{-1}]$ \cite[Proposition~1.16]{Schlichting2004}.
The resulting exact category is denoted $\cN/\cA$.
This quotient category has the universal property
that functors on $\cN/\cA$ correspond bijectively to functors on $\cN$ which vanish on $\cA$ \cite[Proposition~1.16]{Schlichting2004}.

\begin{thm}[{\cite[Theorem~2.10]{Schlichting2004}}]
 \label{thm:schlichtinglocalization}
 Let $\cA$ be an idempotent complete, left $s$-filtering subcategory of the exact category $\cN$.
 Then the sequence $\cA \to \cN \to \cN/\cA$ of exact functors induces a homotopy fiber sequence of spectra
 \[ \bbK^{-\infty}(\cA) \to \bbK^{-\infty}(\cN) \to \bbK^{-\infty}(\cN/\cA). \]
\end{thm}

Finally, recall the \emph{countable envelope} $\cF\cN$ of an idempotent complete exact category $\cN$ \cite[Section~3]{Schlichting2004} (and references therein).
The concrete definition need not concern us here.
It suffices to know that $\cF\cN$ is an exact category which contains $\cN$ as a left $s$-filtering subcategory, and that $\bbK^{-\infty}(\cF\cN)$ is contractible \cite[Lemma~3.2]{Schlichting2004};
the latter claim holds because $\cF\cN$ admits countable coproducts.
Moreover, $\cF\cN$ depends functorially on $\cN$.
Denote by $\cS\cN$ the quotient category $\cF\cN/\cN$.
The category $\cS\cN$ is called the \emph{suspension} of $\cN$.
Write $\cS^n\cN$ for the $n$-fold suspension of $\cN$.
From \cref{thm:schlichtinglocalization}, it follows directly that $\Omega^n\bbK^{-\infty}(\cS^n\cN)$ is naturally equivalent to $\bbK^{-\infty}(\cN)$.
In particular, we have $K_{-n}(\cN) \cong K_0(\Idem(\cS^n\cN))$ for all $n > 0$, where $\Idem(-)$ denotes the idempotent completion functor.

\begin{proof}[Proof of \cref{thm:commute}]
 Let $\{ \cN_i \}_{i \in I}$ be a family of exact categories.
 
 Since the natural map $\bbK^{-\infty}(\cN) \xrightarrow{\sim} \bbK^{-\infty}(\Idem(\cN))$ is an equivalence for every exact category and $\Idem(\prod_{i \in I} \cN_i) \cong \prod_{i \in I} \Idem(\cN_i)$,
 we may assume that $\cN_i$ is idempotent complete for all $i \in I$.
 
 Consider the left $s$-filtering inclusion $\prod_{i \in I} \cN_i \subseteq \cF(\prod_{i \in I} \cN_i)$.
 The various projection functors $\prod_{i \in I} \cN_i \to \cN_j$ induce an exact functor $\cF(\prod_{i \in I} \cN_i) \to \prod_{i \in I} \cF\cN_i$.
 Moreover, the inclusion $\prod_{i \in I} \cN_i \subset \prod_{i \in I} \cF\cN_i$ is left $s$-filtering since it is left $s$-filtering on each factor.
 Since $\cS\cN_i$ is obtained from $\cF\cN_i$ by a calculus of left fractions,
 we can identify $\prod_{i \in I} \cF\cN_i / \prod_{i \in I} \cN_i \cong \prod_{i \in I} \cS\cN_i$.
 Therefore, we have by \cref{thm:schlichtinglocalization} a map of homotopy fiber sequences of spectra:
 \[\begin{tikzcd}
   \bbK^{-\infty}(\prod_{i \in I} \cN_i)\ar[r]\ar[d, "\id"] & \bbK^{-\infty}(\cF(\prod_{i \in I} \cN_i))\ar[r]\ar[d] & \bbK^{-\infty}(\cS(\prod_{i \in I} \cN_i))\ar[d] \\
   \bbK^{-\infty}(\prod_{i \in I} \cN_i)\ar[r] & \bbK^{-\infty}(\prod_{i \in I} \cF\cN_i)\ar[r] & \bbK^{-\infty}(\prod_{i \in I} \cS\cN_i)
  \end{tikzcd} \]
 Since both $\cF(\prod_{i \in I} \cN_i)$ and $\prod_{i \in I} \cF\cN_i$ admit countable coproducts, the $K$-theory of both vanishes and the middle vertical arrow is a $\pi_*$-isomorphism.
 Hence, the right vertical map is a $\pi_*$-isomorphism.
 By induction, it follows that the canonical map
 \[ \bbK^{-\infty}(\cS^n(\prod_{i \in I} \cN_i)) \to \bbK^{-\infty}(\prod_{i \in I} \cS^n\cN_i) \]
 is a $\pi_*$-isomorphism for every family of idempotent complete exact categories.
 
 Let $n > 0$. We have the commutative diagram
 \[\begin{tikzcd}
   K_{-n}(\prod_{i \in I} \cN_i)\ar[r, phantom, "\cong"]\ar[d] &[-20pt] K_0(\Idem(\cS^n(\prod_{i \in I} \cN_i)))\ar[r, "c"]\ar[d] & K_0(\prod_{i \in I} \Idem(\cS^n\cN_i))\ar[dl] \\
   \prod_{i \in I} K_{-n}(\cN_i)\ar[r, phantom, "\cong"] & \prod_{i \in I} K_0(\Idem(\cS^n\cN_i)) &
  \end{tikzcd}\]
 The map $c$ is an isomorphism as we have just discussed.
 Since the diagonal map is an isomorphism by \cref{thm:commute1}, the theorem follows.
\end{proof}

\begin{remark}
 Note that the proof for negative $K$-groups only used that $K_0$ commutes with infinite products,
 which was a direct consequence of \cref{lem:hellerscriterion}.
\end{remark}

\section{The relation to Nenashev's \texorpdfstring{$K_1$}{K1}}\label{sec:nenashev}

The abelian group $K_0(\Omega\cN)$ is not the first algebraic description of $K_1$ of an exact category.
Nenashev gave the following description of $K_1(\cN)$.

\begin{definition}
 Define $K_1^N(\cN)$ as the abelian group generated by binary acyclic complexes of length two
 \[ \bP = \begin{tikzcd} P_2 \ar[r, rightarrowtail, shift right]\ar[r, rightarrowtail, shift left] & P_1\ar[r, twoheadrightarrow, shift right]\ar[r, twoheadrightarrow, shift left] & P_0 \end{tikzcd}\]
 subject to the following relations:
 \begin{enumerate}
  \item If the top and bottom differential of a binary acyclic complex coincide, that complex represents zero.
  \item For any binary acyclic double complex (see \cref{rem:binarydouble_notation})
   \[\begin{tikzcd}
	P_2'\ar[r, rightarrowtail, shift right]\ar[r, rightarrowtail, shift left]\ar[d, rightarrowtail, shift right]\ar[d, rightarrowtail, shift left] & P_1'\ar[r, twoheadrightarrow, shift right]\ar[r, twoheadrightarrow, shift left]\ar[d, rightarrowtail, shift right]\ar[d, rightarrowtail, shift left] & P_0\ar[d, rightarrowtail, shift right]\ar[d, rightarrowtail, shift left] \\
	P_2\ar[r, rightarrowtail, shift right]\ar[r, rightarrowtail, shift left]\ar[d, twoheadrightarrow, shift right]\ar[d, twoheadrightarrow, shift left] & P_1\ar[r, shift right]\ar[r, twoheadrightarrow, shift left]\ar[d, twoheadrightarrow, shift right]\ar[d, twoheadrightarrow, shift left] & P_0\ar[d, twoheadrightarrow, shift right]\ar[d, twoheadrightarrow, shift left] \\
	P_2''\ar[r, rightarrowtail, shift right]\ar[r, rightarrowtail, shift left] & P_1''\ar[r, twoheadrightarrow, shift right]\ar[r, twoheadrightarrow, shift left] & P_0''
	\end{tikzcd} \]
  we have
  \[ [\bP_0] - [\bP_1] + [\bP_2] = [\bP'] - [\bP] + [\bP'']. \]
 \end{enumerate}
\end{definition}

The main result of \cite{Nenashev1998} states that $K_1^N(\cN)$ is isomorphic to $K_1(\cN)$.
By \cref{lem:nena}, regarding a binary acyclic complex of length two as a class in $K_0(\Omega\cN)$ defines a natural homomorphism
\[ \Phi \colon K_1^N(\cN) \to K_0(\Omega\cN), \]
as already remarked in the introduction. In this section, we prove \cref{thm:nenashev_vs_grayson}.  Before doing so, we give the following corollary.

\begin{cor}
	For all $n\geq 1$, the homomorphism $K_0(\Omega_{[0,2]}^n\cN) \to K_0(\Omega^n\cN)$ is a surjection and the homomorphism $K_0(\Omega_{[0,4]}^n\cN) \to K_0(\Omega^n\cN)$ admits a natural section.
\end{cor}
\begin{proof}
By \cref{thm:nenashev_vs_grayson}, $\Phi$ is an isomorphism.
Since $K_0(\Omega_{[0,2]}\cN) \to K_1^N(\cN)$ is a surjection, so is $K_0(\Omega_{[0,2]}\cN)\to K_0(\Omega\cN)$. By \cref{lem:nena}, $\Phi$ factors as
 \[ \Phi \colon K_1^N(\cN) \to K_0(\Omega_{[0,4]}\cN) \to K_0(\Omega\cN). \]
This exhibits $K_0(\Omega\cN)$ as a natural retract of $K_0(\Omega_{[0,4]}\cN)$.
 For $n > 1$, the claim follows as in the proof of \cref{thm:dim7} by induction.
\end{proof}

Hence, \cref{thm:nenashev_vs_grayson} also proves that the algebraic $K$-theory functor commutes with infinite products.

In the remainder of this section, we give a proof of \cref{thm:nenashev_vs_grayson}.
As in the proof of \cref{thm:dim7}, this will be accomplished by producing an explicit formula
that expresses the class of an arbitrary binary acyclic complex in terms of binary acyclic complexes of length two.

Before we start shortening binary acyclic complexes, we make a quick observation about $K_1^N(\cN)$, which we will need later in the argument.

\begin{lemma}\label{lem:signs}
 For any binary acyclic complex of length two, we have
 \[ [ \begin{tikzcd} P_2 \ar[r, rightarrowtail, shift right, "d_2'"']\ar[r, rightarrowtail, shift left, "d_2"] & P_1\ar[r, twoheadrightarrow, shift right, "d_1'"']\ar[r, twoheadrightarrow, shift left, "d_1"] & P_0\end{tikzcd} ]
  = - [ \begin{tikzcd} P_2 \ar[r, rightarrowtail, shift right, "d_2"']\ar[r, rightarrowtail, shift left, "d_2'"] & P_1\ar[r, twoheadrightarrow, shift right, "d_1"']\ar[r, twoheadrightarrow, shift left, "d_1'"] & P_0 \end{tikzcd} ]\in K_1^N(\cN).
 \]
\end{lemma}
\begin{proof}
 This follows directly from applying the defining relations of $K_1^N$ to the binary acyclic double complex
 \[\begin{tikzcd}[column sep = large, baseline=(current bounding box.south)]
    P_2 \oplus P_2 \ar[r, rightarrowtail, shift right, "d_2' \oplus d_2"']\ar[r, rightarrowtail, shift left, "d_2 \oplus d_2'"]\ar[d, rightarrowtail, shift right, "\id"']\ar[d, rightarrowtail, shift left, "\tau_{P_2}"] & P_1\oplus P_1\ar[r, twoheadrightarrow, shift right, "d_1' \oplus d_1"']\ar[r, twoheadrightarrow, shift left, "d_1 \oplus d_1'"]\ar[d, rightarrowtail, shift right, "\id"']\ar[d, rightarrowtail, shift left, "\tau_{P_1}"] & P_0 \oplus P_0\ar[d, rightarrowtail, shift right, "\id"']\ar[d, rightarrowtail, shift left, "\tau_{P_0}"]  \\
    P_2 \oplus P_2 \ar[r, rightarrowtail, shift right, "d_2 \oplus d_2'"']\ar[r, rightarrowtail, shift left, "d_2 \oplus d_2'"] & P_1\oplus P_1\ar[r, twoheadrightarrow, shift right, "d_1 \oplus d_1'"']\ar[r, twoheadrightarrow, shift left, "d_1 \oplus d_1'"] & P_0 \oplus P_0
   \end{tikzcd}\qedhere\]
\end{proof}

Let $\bP:=(P_*, d, d')$ be a binary acyclic complex.
In a first step we will not shorten $\bP$ but produce a complex $\widehat\bP$ representing the same class in $K_0(\Omega\cN)$, which we will then be able to shorten.

Choose factorizations $d_2 \colon P_2 \twoheadrightarrow J \rightarrowtail P_1$ and $d_2' \colon P_2 \twoheadrightarrow K \rightarrowtail P_{1}$.
Since $J$ and $K$ both are the kernel of an admissible epimorphism $P_1 \twoheadrightarrow P_0$, they represent the same class in $K_0(\cN)$.
Therefore, there exist by \cref{lem:hellerscriterion} $A,B,S\in\cN$ and exact sequences 
\[\begin{tikzcd}
A\ar[r, rightarrowtail] & J \oplus S\ar[r, twoheadrightarrow] & B
\end{tikzcd}
\quad\text{and}\quad
\begin{tikzcd}
A\ar[r, rightarrowtail] & K \oplus S\ar[r, twoheadrightarrow] & B.
\end{tikzcd}\]
Let $\bS$ denote the binary acyclic complex
\[\begin{tikzcd}
A\ar[r, shift left]\ar[r, shift right]&K\oplus S\oplus J\ar[r, shift left]\ar[r, shift right]&B\oplus P_1\ar[r, shift left]\ar[r, shift right]&P_0
\end{tikzcd}\]
consisting of top differential
\[\begin{tikzcd}[row sep= tiny]
A\ar[r, rightarrowtail]&K\oplus S\ar[r, twoheadrightarrow]\ar[d, phantom, "\oplus"]& B\ar[d, phantom, "\oplus"]&\\
&J\ar[r, rightarrowtail]&P_1\ar[r, "d_1"]& P_0
\end{tikzcd}\]
and bottom differential
\[\begin{tikzcd}[row sep= tiny]
A\ar[r, rightarrowtail]&J\oplus S\ar[r, twoheadrightarrow]\ar[d, phantom, "\oplus"]& B\ar[d, phantom, "\oplus"]&\\
&K\ar[r, rightarrowtail]&P_1\ar[r, "d_1'"]& P_0.
\end{tikzcd}\]
Let $\bP'$ denote the binary acyclic complex
\[\begin{tikzcd}
\ldots\ar[r, shift left]\ar[r, shift right]&P_3\ar[r, shift left]\ar[r, shift right]&P_2\oplus A\ar[r, shift left]\ar[r, shift right]&J\oplus K\oplus S\ar[r, shift left]\ar[r, shift right]&B
\end{tikzcd}\]
with top differential
\[\begin{tikzcd}[row sep= tiny]
\ldots\ar[r]&P_3\ar[r, "d_3"]&P_{2}\ar[r, twoheadrightarrow, "d_{2}"]\ar[d, phantom, "\oplus"]&J\ar[d, phantom, "\oplus"]&\\
&&A\ar[r, rightarrowtail]&K\oplus S\ar[r, twoheadrightarrow]& B
\end{tikzcd}\]
and bottom differential
\[\begin{tikzcd}[row sep= tiny]
\ldots\ar[r]&P_3\ar[r, "d'_3"]&P_{2}\ar[r, twoheadrightarrow, "d'_{2}"]\ar[d, phantom, "\oplus"]&K\ar[d, phantom, "\oplus"]&\\
&&A\ar[r, rightarrowtail]&J\oplus S\ar[r, twoheadrightarrow]& B
\end{tikzcd}\]
For an object $M\in\cN$ we denote by $\Delta_M$ the binary acyclic complex
\[\begin{tikzcd}
M\ar[r, shift left, "\id_M"]\ar[r, shift right, "\id_M"']&M.
\end{tikzcd}\]
Note that $[\Delta_M]=0\in K_0(\Omega\cN)$.

Consider the following binary acyclic double complex. All differentials written as a single arrow are the identity on the summand appearing in domain and codomain and zero on all other summands. In particular, both differentials agree in this case. The remaining four non-trivial binary acyclic complexes are $\bP\oplus\Delta_B, \bP'$ and $\bS$.
\[\begin{tikzcd}
&&A\ar[d]\ar[r]&A\ar[d, shift left]\ar[d, shift right]&[+5pt]\\
\ldots\ar[r, shift left]\ar[r, shift right]&P_3\ar[r, shift left]\ar[r, shift right]\ar[d]&P_2\oplus A\ar[r, shift left]\ar[r, shift right]\ar[d]&J\oplus K\oplus S\ar[r, shift left]\ar[r, shift right]\ar[d, shift left]\ar[d, shift right]&B\ar[d]\\
\ldots\ar[r, shift left]\ar[r, shift right]&P_3\ar[r, shift left]\ar[r, shift right]&P_2\ar[r, shift left]\ar[r, shift right]&P_1\oplus B\ar[r, shift left]\ar[r, shift right]\ar[d, shift left]\ar[d, shift right]&P_0\oplus B\ar[d]\\
&&&P_0\ar[r]&P_0\\
\end{tikzcd}\]
Applying Nenashev's relation (\cref{lem:nena}) and omitting all summands which are obviously zero, we obtain
\[[\bP] - [\bP'] = [\bS]. \]
Let $\widehat\bP$ denote the binary acyclic complex
\[\begin{tikzcd}
\ldots\ar[r, shift right]\ar[r, shift left]&P_3\ar[r, shift left]\ar[r, shift right]&P_2\oplus J\oplus K\ar[r, shift left]\ar[r, shift right]& J\oplus P_1\oplus K\ar[r, shift left]\ar[r, shift right]&P_0
\end{tikzcd}\]
with top differential
\[\begin{tikzcd}[row sep= tiny]
\ldots \ar[r]&P_3\ar[r, "d_3"]&P_2\ar[r, twoheadrightarrow, "d_2"]\ar[d, phantom, "\oplus"]&J\ar[d, phantom, "\oplus"]&\\
&&J\ar[r, rightarrowtail]\ar[d, phantom, "\oplus"]&P_1\ar[r, "d_1"]\ar[d, phantom, "\oplus"]&P_0\\
&&K\ar[r, "\id_K"]&K&
\end{tikzcd}\]
and bottom differential
\[\begin{tikzcd}[row sep= tiny]
\ldots \ar[r]&P_3\ar[r, "d'_3"]&P_2\ar[r, twoheadrightarrow, "d'_2"]\ar[d, phantom, "\oplus"]&K\ar[d, phantom, "\oplus"]&\\
&&K\ar[r, rightarrowtail]\ar[d, phantom, "\oplus"]&P_1\ar[r, "d'_1"]\ar[d, phantom, "\oplus"]&P_0\\
&&J\ar[r, "\id_J"]&J&
\end{tikzcd}\]
We can build the following binary acyclic double complex involving $\widehat \bP\oplus\Delta_B, \bP'\oplus\Delta_J[1]\oplus\Delta_K[1]$ and $\Delta_J[1]\oplus \bS\oplus \Delta_K[1]$---to recognize the last of these, identify $J\oplus P_1\oplus K\oplus B \cong J\oplus B\oplus P_1 \oplus K$.

\[\begin{tikzcd}
&&A\ar[d]\ar[r]&A\ar[d, shift left]\ar[d, shift right]&[+5pt]\\
\ldots\ar[r, shift left]\ar[r, shift right]&P_3\ar[r, shift left]\ar[r, shift right]\ar[d]&P_2\oplus A\oplus J\oplus K\ar[r, shift left]\ar[r, shift right]\ar[d]&J\oplus K\oplus S\oplus J\oplus K\ar[r, shift left]\ar[r, shift right]\ar[d, shift left]\ar[d, shift right]&B\ar[d]\\
\ldots\ar[r, shift left]\ar[r, shift right]&P_3\ar[r, shift left]\ar[r, shift right]&P_2\oplus J\oplus K\ar[r, shift left]\ar[r, shift right]&J\oplus P_1\oplus K\oplus B\ar[r, shift left]\ar[r, shift right]\ar[d, shift left]\ar[d, shift right]&P_0\oplus B\ar[d]\\
&&&P_0\ar[r]&P_0\\
\end{tikzcd}\]

Applying Nenashev's relation (\cref{lem:nena}) and omitting all summands which are obviously zero, we obtain
\[[\widehat\bP] - [\bP'] = [\bS]\]
and thus $[\widehat{\bP}]=[\bP]$.

Let $\bbQ$ denote the binary acyclic complex
\[\begin{tikzcd}
\ldots\ar[r, shift right]\ar[r, shift left]&P_4\ar[r, shift right]\ar[r, shift left]&P_3\oplus K\oplus J\ar[r, shift left]\ar[r, shift right]&P_2\oplus P_1\oplus J\oplus K\ar[r, shift left]\ar[r, shift right]&J\oplus P_0\oplus K
\end{tikzcd}\]
with top differential
\[\begin{tikzcd}[row sep= tiny]
\ldots \ar[r]&P_4\ar[r, "d_4"]&P_3\ar[r, "d_3"]\ar[d, phantom, "\oplus"]&P_2\ar[r, twoheadrightarrow, "d_2"]\ar[d, phantom, "\oplus"]&J\ar[d, phantom, "\oplus"]\\
&&K\ar[r, rightarrowtail]\ar[d, phantom, "\oplus"]&P_1\ar[r, "d_1'"]\ar[d, phantom, "\oplus"]&P_0\ar[dd, phantom, "\oplus"]\\
&&J\ar[r, "\id_J"]&J\ar[d, phantom, "\oplus"]&\\
&&&K\ar[r, "\id_K"]&K
\end{tikzcd}\]
and bottom differential
\[\begin{tikzcd}[row sep= tiny]
\ldots \ar[r]&P_4\ar[r, "d'_4"]&P_3\ar[r, "d'_3"]\ar[d, phantom, "\oplus"]&P_2\ar[r, twoheadrightarrow, "d'_2"]\ar[d, phantom, "\oplus"]&K\ar[d, phantom, "\oplus"]\\
&&J\ar[r, rightarrowtail]\ar[d, phantom, "\oplus"]&P_1\ar[r, "d_1"]\ar[d, phantom, "\oplus"]&P_0\ar[dd, phantom, "\oplus"]\\
&&K\ar[r, "\id_K"]&K\ar[d, phantom, "\oplus"]&\\
&&&J\ar[r, "\id_J"]&J
\end{tikzcd}\]
Let us fix the following notation:
If $M$ is an object containing $N$ as a direct summand, denote by $e_N$ the obvious idempotent $M \to M$ whose image is $N$.

Consider the following double complex involving $\widehat \bP$ and $\bbQ[1]$. Note that only the rows are acyclic. This suffices to see that the total complex shifted down by one represents the same class as $[\bbQ]+[\widehat\bP]$.
{\small\[\begin{tikzcd}
\ldots\ar[r, shift left]\ar[r, shift right]&P_4\ar[d]\ar[r, shift left]\ar[r, shift right]&P_3\ar[d]\ar[r, shift left]\ar[r, shift right]&P_2\oplus J\oplus K\ar[d, "e_{P_2}"]\ar[r, shift left]\ar[r, shift right]&J\oplus P_1\oplus K\ar[d, shift right, "e_J"']\ar[d, shift left, "e_K"]\ar[r, shift left]\ar[r, shift right]&P_0\\
\ldots\ar[r, shift left]\ar[r, shift right]&P_4\ar[r, shift left]\ar[r, shift right]&P_3\oplus K\oplus J\ar[r, shift left]\ar[r, shift right]&P_2\oplus P_1\oplus J\oplus K\ar[r, shift left]\ar[r, shift right]&J\oplus P_0\oplus K&
\end{tikzcd}\]}
Let $\bT$ denote the total complex shifted down by one. Then $\bT$ is
{\footnotesize\[\begin{tikzcd}
\ldots\ar[r, shift left]\ar[r, shift right]&P_4\oplus P_3\ar[r, shift left]\ar[r, shift right]& \begin{array}{c}
P_3\oplus K\oplus J\oplus\\ P_2\oplus J\oplus K
\end{array}
\ar[r, shift left]\ar[r, shift right]&
\begin{array}{c}
P_2\oplus P_1\oplus J\oplus \\K\oplus J\oplus P_1\oplus K
\end{array}\ar[r, shift left]\ar[r, shift right]& J\oplus P_0\oplus K\oplus P_0.
\end{tikzcd}\]}
Assume that $\bP$ was supported on $[0,m]$, then $\bT$ admits a projection onto $\Delta_{P_m}[m-1]$. The kernel of this projection admits a projection to $\Delta_{P_{m-1}}[m-2]$ and so on until we take the kernel of the projection to $\Delta_{P_2}[1]$. The remaining acyclic binary complex $\bT'$ is
\[\begin{tikzcd}
K\oplus J\oplus J\oplus K\ar[r, shift left]\ar[r, shift right]& P_1\oplus J\oplus K\oplus J\oplus P_1\oplus K\ar[r, shift left]\ar[r, shift right]& J\oplus P_0\oplus K\oplus P_0
\end{tikzcd}\]
with top differential
\[\begin{tikzcd}[row sep=tiny, column sep=large]
		& P_1\ar[rdd]	& 	\\
K\ar[ru] 	& J		& J 	\\
J\ar[ru] 	& K\ar[rd]	& P_0	\\
J\ar[rd]	& J\ar[ruu] 	& K	\\
K\ar[rd] 	& P_1\ar[r] 	& P_0	\\
 		& K		& 				
\end{tikzcd}\]
and bottom differential
\[\begin{tikzcd}[row sep= tiny]
		& P_1\ar[rdd]	& 	\\
K\ar[rd] 	& J\ar[r]	& J 	\\
J\ar[ruu] 	& K		& P_0	\\
J\ar[r]		& J	 	& K	\\
K\ar[r] 	& P_1\ar[r]	& P_0	\\
 		& K\ar[ruu]	& 				
\end{tikzcd}\]
It follows that $[\bP]=[\bT']-[\bQ]$.
Since $\bT'$ is supported on $[0,2]$ and $\bbQ$ has length one shorter than $\bP$, iterating this argument already shows that $K_0(\Omega_{[0,2]}\cN)\to K_0(\Omega\cN)$ is surjective.

\begin{rem}
	\label{rem:grayson}
	The idea to use the complexes $\widehat{\bP},\bT$ and $\bT'$ is from the aforementioned, unpublished result of Grayson. He uses a different argument to compute $[\widehat{\bP}]-[\bP]$ which only shows that it is contained in the image of $\Phi$. 
\end{rem}

We now want to simplify $\bT'$. Let $\bT_{triv}'$ denote the binary acyclic complex whose underlying graded object is that of $\bT'$, but with both differentials equal to the top differential of $\bT'$.
Then the following diagram, where the upper row is $\bT'$ and the second row is $\bT'_{triv}$, commutes.
\[\begin{tikzcd}
K\oplus J\oplus J\oplus K\ar[d, shift right, "\id"']\ar[d, shift left, "\tau_K\oplus\tau_J"]\ar[r, shift left]\ar[r, shift right]&P_1\oplus J\oplus K\oplus J\oplus P_1\oplus K\ar[d, shift right, "\id"']\ar[d, shift left, "\tau_{P_1}\oplus\tau_J\oplus\tau_K"]\ar[r, shift left]\ar[r, shift right]& J\oplus P_0\oplus K\oplus P_0\ar[d, shift right, "\id"']\ar[d, shift left, "\tau_{P_0}"] \\
K\oplus J\oplus J\oplus K\ar[r, shift left]\ar[r, shift right]&P_1\oplus J\oplus K\oplus J\oplus P_1\oplus K\ar[r, shift left]\ar[r, shift right]& J\oplus P_0\oplus K\oplus P_0
\end{tikzcd}\]
Both differentials of $\bT'_{triv}$ agree and thus it represents the trivial class. Since $\tau_K$ and $\tau_J$ are of order two, we conclude from \cref{lem:nena} that
\[ [\bT'] = [ \begin{tikzcd} P_0\oplus P_0\ar[r, shift left, "\id"]\ar[r, shift right, "\tau_{P_0}"']&P_0\oplus P_0 \end{tikzcd}]-[\begin{tikzcd} P_1\oplus P_1\ar[r, shift left, "\id"]\ar[r, shift right, "\tau_{P_1}"']&P_1\oplus P_1 \end{tikzcd}]. \]
Since $J\rightarrowtail P_1\twoheadrightarrow P_0$ is exact, this is the same as
$[\begin{tikzcd} J\oplus J\ar[r, shift left, "\id"]\ar[r, shift right, "\tau_{J}"']&J\oplus J \end{tikzcd}]$.
This shows that
\begin{equation}\label{eq:shortening}
 [\bP]=[\begin{tikzcd} J\oplus J\ar[r, shift left, "\id"]\ar[r, shift right, "\tau_{J}"']&J\oplus J \end{tikzcd}]-[\bbQ].
\end{equation}
We are now going to iterate this argument.
Choose factorizations $d_n \colon P_n \twoheadrightarrow J_{n-1} \rightarrowtail P_{n-1}$ and $d_n' \colon P_n \twoheadrightarrow K_{n-1} \rightarrowtail P_{n-1}$ for all $n \geq 2$
such that $J_n \rightarrowtail P_n \twoheadrightarrow J_{n-1}$ and $K_n \rightarrowtail P_n \twoheadrightarrow K_{n-1}$ are exact for all $n$. Set $J_0 := P_0$ and $K_0 := P_0$.
For any natural number $k$, fix the following auxiliary notation:
\begin{align*}
 J_\odd^k &:= \bigoplus_{\substack{n \leq k,\\ n\ \text{odd}}} J_n, &J_\even^k &:= \bigoplus_{\substack{2 \leq n \leq k,\\ n\ \text{even}}} J_n, &J_{\even,0}^k &:= \bigoplus_{\substack{n \leq k,\\ n\ \text{even}}} J_n,\\
 K_\odd^k &:= \bigoplus_{\substack{n \leq k,\\ n\ \text{odd}}} K_n, &K_\even^k &:= \bigoplus_{\substack{2 \leq n \leq k,\\ n\ \text{even}}} K_n, &K_{\even,0}^k &:= \bigoplus_{\substack{n \leq k,\\ n\ \text{even}}} K_n, \\
 P_\odd^k &:= \bigoplus_{\substack{n \leq k,\\ n\ \text{odd}}} P_n, & P_\even^k &:= \bigoplus_{\substack{2 \leq n \leq k,\\ n\ \text{even}}} P_n.
\end{align*}
First of all, we define for every natural number $k$ a binary acyclic complex $\bP_k$ of the form
{\tiny \[\begin{tikzcd}
	\ldots\ar[r, shift left]\ar[r, shift right]&P_{k+3}\ar[r, shift left]\ar[r, shift right]&
	\begin{array}{c}
	P_{k+2}\oplus \\ \bigoplus_{n=1}^{k}(J_n\oplus K_n)
	\end{array}\ar[r, shift left]\ar[r, shift right]&
	\begin{array}{c}
	\bigoplus_{n=1}^{k+1} P_{n} \oplus \\ \bigoplus_{n=1}^{k} (J_n \oplus K_n)
	\end{array}\ar[r, shift left]\ar[r, shift right]&
	\begin{array}{c}
	P_0 \oplus \\ \bigoplus_{n=1}^{k} (J_n \oplus K_n)
	\end{array}\end{tikzcd} \]}
For even natural numbers $k$, we equip $\bP_k$ with the top differential
\[\begin{tikzcd}[row sep=tiny]
	\ldots\ar[r]&P_{k+3}\ar[r, "d_{k+3}"]&P_{k+2}\ar[r, "d_{k+2}"]\ar[d, phantom, "\oplus"]&P_{k+1}\ar[r, "d_{k+1}"]\ar[d, phantom, "\oplus"]&J_{k}\ar[ddd, phantom, "\oplus"]\\
	&&K_\odd^k\ar[r]\ar[d, phantom, "\oplus"]&K_\odd^k\ar[d, phantom, "\oplus"]&\\
	&&J_\even^{k}\ar[r]\ar[ddd, phantom, "\oplus"]&J_\even^{k}\ar[d, phantom, "\oplus"]&\\
	&&&K_\even^{k}\ar[r]\ar[d, phantom, "\oplus"]&K_\even^{k}\ar[d, phantom, "\oplus"]\\
	&&&J_\odd^k\ar[r]\ar[d, phantom, "\oplus"]&J_\odd^k\ar[d, phantom, "\oplus"]\\
	&&J_\odd^k\ar[r]\ar[d, phantom, "\oplus"]&P_\odd^k\ar[r]\ar[d, phantom, "\oplus"]&J_{\even,0}^{k-1}\ar[d, phantom, "\oplus"]\\
	&&K_\even^{k}\ar[r]&P_\even^{k}\ar[r]&K_\odd^{k}
	\end{tikzcd}\]
and bottom differential
\[\begin{tikzcd}[row sep=tiny]
	\ldots\ar[r]&P_{k+3}\ar[r, "d'_{k+3}"]&P_{k+2}\ar[r, "d'_{k+2}"]\ar[d, phantom, "\oplus"]&P_{k+1}\ar[r, "d'_{k+1}"]\ar[d, phantom, "\oplus"]&K_{k}\ar[ddd, phantom, "\oplus"]\\
	&&J_\odd^k\ar[r]\ar[d, phantom, "\oplus"]&J_\odd^k\ar[d, phantom, "\oplus"]&\\
	&&K_\even^{k}\ar[r]\ar[ddd, phantom, "\oplus"]&K_\even^{k}\ar[d, phantom, "\oplus"]&\\
	&&&J_\even^{k}\ar[r]\ar[d, phantom, "\oplus"]&J_\even^{k}\ar[d, phantom, "\oplus"]\\
	&&&K_\odd^k\ar[r]\ar[d, phantom, "\oplus"]&K_\odd^k\ar[d, phantom, "\oplus"]\\
	&&K_\odd^k\ar[r]\ar[d, phantom, "\oplus"]&P_\odd^k\ar[r]\ar[d, phantom, "\oplus"]&K_{\even,0}^{k-1}\ar[d, phantom, "\oplus"]\\
	&&J_\even^{k}\ar[r]&P_\even^{k}\ar[r]&J_\odd^{k}
	\end{tikzcd}\]
Note that $\bP_0$ is precisely the complex $\bP$.

For odd natural numbers $k$, we equip $\bP_k$ with the top differential
\[\begin{tikzcd}[row sep=tiny]
	\ldots\ar[r]&P_{k+3}\ar[r, "d_{k+3}"]&P_{k+2}\ar[r, "d_{k+2}"]\ar[d, phantom, "\oplus"]&P_{k+1}\ar[r, "d_{k+1}"]\ar[d, phantom, "\oplus"]&J_{k}\ar[ddd, phantom, "\oplus"]\\
	&&J_\odd^k\ar[r]\ar[d, phantom, "\oplus"]&J_\odd^k\ar[d, phantom, "\oplus"]&\\
	&&K_\even^{k}\ar[r]\ar[ddd, phantom, "\oplus"]&K_\even^{k}\ar[d, phantom, "\oplus"]&\\
	&&&J_\even^{k}\ar[r]\ar[d, phantom, "\oplus"]&J_\even^{k}\ar[d, phantom, "\oplus"]\\
	&&&K_\odd^k\ar[r]\ar[d, phantom, "\oplus"]&K_\odd^k\ar[d, phantom, "\oplus"]\\
	&&K_\odd^k\ar[r]\ar[d, phantom, "\oplus"]&P_\odd^k\ar[r]\ar[d, phantom, "\oplus"]&K_{\even,0}^{k}\ar[d, phantom, "\oplus"]\\
	&&J_\even^{k}\ar[r]&P_\even^{k}\ar[r]&J_\odd^{k-1}
	\end{tikzcd}\]
and bottom differential
\[\begin{tikzcd}[row sep=tiny]
	\ldots\ar[r]&P_{k+3}\ar[r, "d'_{k+3}"]&P_{k+2}\ar[r, "d'_{k+2}"]\ar[d, phantom, "\oplus"]&P_{k+1}\ar[r, "d'_{k+1}"]\ar[d, phantom, "\oplus"]&K_{k}\ar[ddd, phantom, "\oplus"]\\
	&&K_\odd^k\ar[r]\ar[d, phantom, "\oplus"]&K_\odd^k\ar[d, phantom, "\oplus"]&\\
	&&J_\even^{k}\ar[r]\ar[ddd, phantom, "\oplus"]&J_\even^{k}\ar[d, phantom, "\oplus"]&\\
	&&&K_\even^{k}\ar[r]\ar[d, phantom, "\oplus"]&K_\even^{k}\ar[d, phantom, "\oplus"]\\
	&&&J_\odd^k\ar[r]\ar[d, phantom, "\oplus"]&J_\odd^k\ar[d, phantom, "\oplus"]\\
	&&J_\odd^k\ar[r]\ar[d, phantom, "\oplus"]&P_\odd^k\ar[r]\ar[d, phantom, "\oplus"]&J_{\even,0}^{k}\ar[d, phantom, "\oplus"]\\
	&&K_\even^{k}\ar[r]&P_\even^{k}\ar[r]&K_\odd^{k-1}
	\end{tikzcd}\]
Note that $\bP_1$ is precisely the complex $\bQ$ appearing in \eqref{eq:shortening}.
Moreover, if $k$ is sufficiently large so that $P_n \cong 0$ for all $n > k$, then $\bP_{k+1}$ is obtained from $\bP_k$ by interchanging the top and bottom differential.

For every $k$ let $\bQ_k$ denote the complex obtained from $\bP_k$ by the same procedure as $\bQ$ is obtained from $\bP$.

Suppose now that $k$ is odd. Substituting appropriately in \eqref{eq:shortening}, we obtain the equation
\[ [\bP_k] = [ \begin{tikzcd} X\oplus X\ar[r, shift left, "\id"]\ar[r, shift right, "\tau_{X}"']&X\oplus X \end{tikzcd} ] - [\bQ_k] \in K_0(\Omega\cN), \]
where $X$ denotes the kernel of the first top differential of $\bP_k$.

By the definition of the binary acyclic complex $\bP_k$, we may choose
\[ X := J_{k+1} \oplus \bigoplus_{n=1}^k(J_n \oplus K_n). \]
As in the proof of \cref{prop:split}, since $J_n$ and $K_n$ represent the same class in $K_0$ for all $n$, we have by \cref{lem:ordertwo}
\begin{eqnarray}
&[\begin{tikzcd}
(J_n\oplus K_n)\oplus (J_n\oplus K_n)\ar[r, shift left, "\id"]\ar[r, shift right, "\tau_{J_n\oplus K_n}"']&(J_n\oplus K_n)\oplus(J_n\oplus K_n)
\end{tikzcd}]\nonumber\\
&\label{eq:tauZ=0}=2\cdot[\begin{tikzcd}
J_n\oplus J_n\ar[r, shift left, "\id"]\ar[r, shift right, "\tau_{J_n}"']&J_n\oplus J_n
\end{tikzcd} ]=0.
\end{eqnarray}
Therefore,
\[[ \begin{tikzcd} X\oplus X\ar[r, shift left, "\id"]\ar[r, shift right, "\tau_{X}"']&X\oplus X \end{tikzcd} ]=[ \begin{tikzcd} J_{k+1}\oplus J_{k+1}\ar[r, shift left, "\id"]\ar[r, shift right, "\tau_{J_{k+1}}"']&J_{k+1}\oplus J_{k+1} \end{tikzcd} ].\]
Similarly,
\[ Y := K_{k+1} \oplus \bigoplus_{n=1}^k(J_n \oplus K_n) \]
is the kernel of the first bottom differential of $\bP_k$.
Note that the complement of $J_{k+1}$ in $X$ and the complement of $K_{k+1}$ in $Y$ are the same; let $Z$ denote that complement. 

Unwinding the definition of $\bQ_k$, we see that, up to automorphisms flipping the two copies of $Z$ in the three lowest degrees of $\bQ_k$,
$\bQ_k$ coincides with the sum of $\bP_{k+1}$ with some complexes in the image of the diagonal functor $\Delta$.
Since, by \eqref{eq:tauZ=0}, $[ \begin{tikzcd} Z\oplus Z\ar[r, shift left, "\id"]\ar[r, shift right, "\tau_{Z}"']&Z\oplus Z \end{tikzcd} ] = 0$,
we see that $[\bQ_k] = [\bP_{k+1}]$. Hence,
\[ [\bP_k] = [ \begin{tikzcd} J_{k+1}\oplus J_{k+1}\ar[r, shift left, "\id"]\ar[r, shift right, "\tau_{J_{k+1}}"']&J_{k+1}\oplus J_{k+1} \end{tikzcd} ] - [\bP_{k+1}]. \]
The argument for $k$ even is completely analogous.
Therefore, we have for every $k \geq 0$ the equation
\[ [\bP] = x(\bP,k) := (-1)^{k}[\bP_{k}] + \sum_{n=1}^{k} [\begin{tikzcd}
J_{n}\oplus J_{n}\ar[r, shift left, "\id"]\ar[r, shift right, "\tau_{J_{n}}"']&J_{n}\oplus J_{n}
\end{tikzcd}]. \]

\begin{proof}[Proof of \cref{thm:nenashev_vs_grayson}]
 Define a map $\Psi \colon K_0(\Omega\cN)\to K_1^N(\cN)$ by the rule
 \[ [ \bP ] \mapsto x(\bP, k(\bP)), \]
 where $k(\bP)$ is defined to be
 \[ k(\bP) := \min \{ n \in \N \mid P_{n'} \cong 0\ \text{for all}\ n' > n \}. \]
 We have to show that this is a well-defined homomorphism.
 By our definition of $k(\bP)$, the complex $\bP_{k(\bP)}$ has length two. 
 
 Let $\bP' \rightarrowtail \bP \twoheadrightarrow \bP''$ be an exact sequence of binary acyclic complexes.
 Evidently, $x(\bP',k(\bP)) + x(\bP'',k(\bP)) = x(\bP,k(\bP))$.
 Note that $k(\bP'), k(\bP'') \leq k(\bP)$ and at least one of $k(\bP')$ and $k(\bP'')$ equals $k(\bP)$.
 If $k(\bP') = k(\bP) = k(\bP'')$, we already have $\Psi([\bP']) + \Psi([\bP'']) = \Psi([\bP])$.
 Suppose $k(\bP') < k(\bP)$.
 Then $\bP_{k(\bP)}$ arises from $\bP_{k(\bP')}$ by interchanging the role of top and bottom differential $k(\bP) - k(\bP')$ times.
 Since interchanging the top and bottom differential results only in a change of sign (\cref{lem:signs}) and $J'_n \cong 0$ for $n > k(\bP')$,
 we have $x(\bP',k(\bP')) = x(\bP',k(\bP))$.
 The case $k(\bP'') < k(\bP)$ is analogous.
  
 Suppose now that $\bP$ lies in the image of the diagonal functor $\Delta \colon C^q\cN \to B^q\cN$.
 Then we may choose $J_n = K_n$ for all $n$.
 In this case, the top and bottom differential of $\bP_{k(\bP)}$ are isomorphic.
 However, the two differentials do not agree on the nose but only after flipping all appearing $K_n=J_n$.
 Since each one of these appears three times in $\bP_{k(\bP)}$, applying the Nenashev relation we see that
 \[ [ \bP_{k(\bP)} ] = \sum_{n=1}^{k(\bP)}[\begin{tikzcd} J_{n}\oplus J_{n}\ar[r, shift left, "\id"]\ar[r, shift right, "\tau_{J_{n}}"']&J_{n}\oplus J_{n} \end{tikzcd}]. \]
 Consequently, $\Psi([\bP]) = 0$ by \cref{lem:ordertwo}. 
 This shows that $\Psi$ is a well-defined homomorphism $K_0(\Omega\cN) \to K_1^N(\cN)$.
 
 Our previous discussion implies that $\Phi \circ \Psi = \id_{K_0(\Omega\cN)}$.
 What is left to do is to show that $\Psi \circ \Phi = \id_{K_1^N(\cN)}$.
 To do so, it suffices to establish \cref{eq:shortening} in $K_1^N(\cN)$ for all binary acyclic complexes of length two.
 Let 
 \[ \bP = \begin{tikzcd} P_2 \ar[r, rightarrowtail, shift right, "d_2'"']\ar[r, rightarrowtail, shift left, "d_2"] & P_1\ar[r, twoheadrightarrow, shift right, "d_1'"']\ar[r, twoheadrightarrow, shift left,"d_1"] & P_0 \end{tikzcd}\]
 be a binary acyclic complex of length two.
 Then $\bP_1$ is the binary acyclic complex
 \[ \begin{tikzcd} P_2 \oplus P_2 \ar[r, rightarrowtail, shift right]\ar[r, rightarrowtail, shift left] & P_2 \oplus P_2 \oplus P_2 \oplus P_1\ar[r, twoheadrightarrow, shift right]\ar[r, twoheadrightarrow, shift left] & P_2 \oplus P_2 \oplus P_0 \end{tikzcd} \]
 with the following top and bottom differentials:
 \[\begin{tikzcd}[row sep=tiny]
    & P_2\ar[r] & P_2 \\
    P_2\ar[r] & P_2 & \\
    & P_2\ar[r] & P_2 \\
    P_2\ar[r, "d_2'"] & P_1\ar[r, "d_1'"] & P_0
   \end{tikzcd}
   \qquad\qquad
   \begin{tikzcd}[row sep=tiny]
    & P_2\ar[rdd] & P_2 \\
    P_2\ar[rdd, "d_2"] & P_2\ar[ru] & \\
    & P_2 & P_2 \\
    P_2\ar[ru] & P_1\ar[r, "d_1"] & P_0
   \end{tikzcd}\]
  %%%% 
  Consider the following binary acyclic double complex where the upper row is $\bP_1$, the lower row is $\bP$ with switched differentials plus $\Delta_{P_2}\oplus \Delta_{P_2}\oplus\Delta_{P_2}[1]$ and the middle vertical map $\tau_{P_2}$ denotes the flip of the second and third summand.
  \[\begin{tikzcd}
  P_2 \oplus P_2 \ar[d, shift left, "\tau_{P_2}"]\ar[d, shift right, "\id"']\ar[r, rightarrowtail, shift right]\ar[r, rightarrowtail, shift left] & P_2 \oplus P_2 \oplus P_2 \oplus P_1\ar[d, shift left, "\tau_{P_2}"]\ar[d, shift right, "\id"']\ar[r, twoheadrightarrow, shift right]\ar[r, twoheadrightarrow, shift left] & P_2 \oplus P_2 \oplus P_0\ar[d, shift left, "\tau_{P_2}"]\ar[d, shift right, "\id"']\\
  P_2 \oplus P_2 \ar[r, rightarrowtail, shift right]\ar[r, rightarrowtail, shift left] & P_2 \oplus P_2 \oplus P_2 \oplus P_1\ar[r, twoheadrightarrow, shift right]\ar[r, twoheadrightarrow, shift left] & P_2 \oplus P_2 \oplus P_0
  \end{tikzcd}\]
  Therefore, we obtain using \cref{lem:ordertwo} and \cref{lem:signs}
  \[[\bP]=[\begin{tikzcd}
  P_2\oplus P_2\ar[r, shift left, "\id"]\ar[r, shift right, "\tau_{P_2}"']&P_2\oplus P_2
  \end{tikzcd}]-[\bP_1].\]
   %%%%
 This finishes the proof.
\end{proof}

\bibliographystyle{amsalpha}
\bibliography{shortening}

%%%%%%%%

\end{document}